\documentclass{article}
\usepackage{amsfonts}
\usepackage{amsmath}
\usepackage[ignoreall]{geometry}

\newtheorem{theorem}{Theorem}

\newtheorem{example}[theorem]{Example}

\newtheorem{lemma}[theorem]{Lemma}

\newtheorem{proposition}[theorem]{Proposition}
\newtheorem{remark}[theorem]{Remark}

\newenvironment{proof}[1][Proof]{\noindent\textbf{#1.} }{\ \rule{0.5em}{0.5em}}
\input{tcilatex}

\begin{document}

\title{\textbf{Metrizability of $SO\left( 3\right) $-invariant connections:\
Riemann versus Finsler}}
\author{\textbf{\ Nicoleta Voicu}$^{\mathbf{1}}$\textbf{{\ and Salah Gomaa
Elgendi$^{2,3}$ }} \\
$^{1}${\small Department of Mathematics and Computer Science,}\\
{\small Transilvania University, Brasov, Romania; }\\
{$^{2}$}{\small Department of Mathematics, Faculty of Science,}\\
{\small Islamic University of Madinah, Kingdom of Saudi Arabia }\\
{$^{3}$}{\small Department of Mathematics, Faculty of Science,}\\
{\small Benha University, Egypt,}\\
{\small e-mails: nico.voicu@unitbv.ro, salah.ali@fsc.bu.edu.eg}}
\date{}
\maketitle

\begin{abstract}
For a torsion-free affine connection on a given manifold, which does not
necessarily arise as the Levi-Civita connection of any pseudo-Riemannian
metric, it is still possible that it corresponds in a canonical way to a
Finsler structure; this property is known as Finsler (or Berwald-Finsler)
metrizability.

In the present paper, we clarify, for 4-dimensional $SO(3)$-invariant,
Berwald-Finsler metrizable connections, the issue of the existence of an
affinely equivalent pseudo-Riemannian structure. In particular, we find all
classes of \ $SO\left( 3\right) $-invariant connections which are not
Levi-Civita connections for any pseudo-Riemannian metric - hence, are
non-metric in a conventional way - but can still be metrized by $SO(3)$%
-invariant Finsler functions. The implications for physics, together with
some examples are briefly discussed.
\end{abstract}

\section{Introduction}

In the search of extensions of general relativity that may explain
geometrically the issues arising at either the largest scales (e.g., the
accelerated expansion of the universe) or at the smallest scales (tensions
with quantum mechanics), one needs to relax, in a way or another, the
"Riemannian constraint" upon the geometry underlying our spacetime models. A
most popular and well-studied choice is represented by metric-affine
geometry, based on a Riemannian arc length and an independent connection, 
\cite{Blagojevic:2012}, \cite{Hehl:1976}, \cite{Hehl:1994}, \cite{CANTATA}.

Another route is to relax our very definition of the spacetime interval.
This is precisely done by \textit{Finsler geometry}, which is the geometry
of a manifold equipped with a most general notion of
parametrization-invariant notion of arc length -- not necessarily arising
from a pseudo-scalar product of tangent vectors, but just, roughly speaking,
from a (pseudo)-norm -- and has been so far proven to be an excellent
mathematical framework for modified dispersion relations appearing in
quantum gravity phenomenology \cite{Amelino-Camelia}, \cite{Girelli}, \cite%
{Pfeifer-Finsler-physics}, \cite{Raetzel}, for theories exhibiting
broken/deformed Lorentz invariance, \cite{Bogoslovsky-Goenner}, \cite%
{Edwards}, \cite{Foster}, \cite{Gibbons}, \cite{Letizia}, for the
description of the gravitational field of kinetic gases \cite{kinetic gas}, 
\cite{kinetic gas essay}, \cite{math-foundations}, or of propagation of
fields in media, \cite{Markvorssen}, \cite{Perlick}, \cite{Yajima}.

In Finslerian models of gravity, arc length is physically interpreted as
proper time, whereas free fall trajectories are modeled as Finslerian
geodesics. The latter are obtained in a largely similar manner to Riemannian
geometry; yet, the natural and canonical generalization of the Levi-Civita
connection, determined by the Finslerian geodesic equation, is generally not
a linear (affine) connection on spacetime, but just a \textit{nonlinear }one.

Finsler spaces for which this canonical connection is an affine one, are
known as \textit{Berwald spaces. }These are, in a sense, the closest Finsler
spaces to pseudo-Riemannian ones. Actually, one of the first questions that
arise, for a Berwald space, is to what extent is its geodesic structure
different from a Riemannian one - more precisely, if its canonical affine
connection (which is, by definition, symmetric) is actually not the
Levi-Civita connection of some pseudo-Riemannian metric. Depending on the
answer, there exist two possible scenarios:

\begin{enumerate}
\item If the connection is also Riemann metrizable, then spacetime geodesics
are equally well described by the pseudo-Riemannian and by the Finslerian
metric; the difference between the predictions of the two metric models is
identifiable in the form of \textit{time delays }due to a different
measurement of proper time. This is, for instance, trivially the case in
deformed special relativity, where geodesics are straight lines, arising
both as geodesics of the Minkowski metric and of a flat Finslerian metric.

\item If the connection is not Riemann metrizable, Finslerian geodesics are
affine geodesics that cannot be re-assigned to any\ pseudo-Riemannian
metric. Since basically in all gravitational applications, a fiducial
Riemannian metric (e.g., Schwarzschild, or
Friedmann-Lemaitre-Robertson-Walker, depending on the problem at hand) is
still available, such Berwald-Finsler models can be also interpreted as
metric-affine ones with nonmetricity; their distinctive feature is that
their affine geodesics are \textit{variational}, as explained below.
\end{enumerate}

In general metric-affine theories with nonmetricity, metric geodesics
("shortest paths") and affine geodesics ("straightest paths") do not
coincide; moreover, the latter are typically not variational, that is, they
do not arise as Euler-Lagrange equations of any point particle Lagrangian.
The Berwald metrizability property, if present, ensures precisely that
affine geodesics \textit{do }arise, in fact, as extremals of an arc length
Lagrangian - just, not a Riemannian, but a Finslerian one.

\bigskip

The existence of affine connections that are metrizable by Lorentzian
signature Finsler metrics, but not pseudo-Riemann metrizable, first proven
by a concrete example in \cite{Berwald-non-metrizability}, is actually,
quite surprising; classical (smooth, positive) definite Berwald-Finsler
structures had been for long known by a famous theorem by Szab\'{o}, \cite%
{Szabo} to always be Riemannian metrizable. A legitimate question is, thus,
how many such connections do actually exist.

\bigskip

In the present paper, we give a complete answer to the above question, for
the class of spherically symmetric Berwald-Finsler structures on subsets of $%
\mathbb{R}^{4}$.

To this aim, we first find, in Theorem \ref{thm:main}, the necessary and
sufficient conditions for the pseudo-Riemann metrizability of a given
non-Riemannian Berwald-Finsler structure: the symmetry of the connection
Ricci tensor (which was known, \cite{Berwald-non-metrizability}, \cite%
{Heefer-thesis}, \cite{Heefer-m-Kropina}, to be necessary), together with a
non-maximal dimension of the vertical part of the corresponding holonomy
algebra. These two conditions are independent, as proven by concrete
examples.

Then, using the above result and a previous local classification of $%
SO\left( 3\right) $-invariant, non-Riemannian Berwald structures, \cite%
{Voicu}, we are able to determine, for each of the five existing classes,
its Riemann metrizability and, in the affirmative case, to concretely
construct an affinely equivalent Riemannian metric.

We find that two of these five classes - for which we provide concrete
examples - are never Riemann metrizable.

\bigskip

The paper is structured as follows. In Section \ref{sec:preliminaries}, we
briefly present the known notions and results to be used subsequently; then,
in Section \ref{sec:main}, we state our main result and give an outline of
its proof; Section 4 presents the details of this proof, together with some
concrete examples. An analysis of the obtained results, as well as the
remaining open questions and plans for future work are presented in the
Conclusion. Finally, some technical results in \cite{Voicu}, which are used
in the proof, are listed in the Appendix. For calculations in the examples,
we use Maple software or the Finsler package \cite{NF-Package}.

\section{Preliminaries \label{sec:preliminaries}}

This section briefly reviews the basic notions of (pseudo-)Finsler geometry
and, in particular, of Berwald-type geometry, needed in the following.

Consider a smooth manifold $M$, of arbitrary (for the moment) dimension $n.$
On its tangent bundle $\left( TM,\pi ,M\right) ,$ we denote by $(\pi
^{-1}\left( U\right) ,(x^{a},\dot{x}^{a})),$ the naturally induced chart by
an arbitrarily fixed chart $(U,(x^{a}))$ on $M;$ that is, for a vector $\dot{%
x}\in T_{x}M$, $\dot{x}^{a}$ denote its coordinates in the natural local
basis $\{\partial _{a}\}$ of $T_{x}M$. If there is no risk of confusion, we
will abuse the notation by skipping the indices of the coordinates, i.e., we
will refer to $(x^{a},\dot{x}^{a})$ briefly as~$(x,\dot{x})$. The natural
local coordinate basis of $T_{(x,\dot{x})}TM$ is given by $\{\partial _{a}=%
\tfrac{\partial }{\partial x^{a}},\dot{\partial}_{a}=\tfrac{\partial }{%
\partial \dot{x}^{a}}\}$ and the local coordinate basis of $T_{(x,\dot{x}%
)}^{\ast }TM$ is $\{dx^{a},d\dot{x}^{a}\}$. By $\mathcal{X}\left( M\right) ,$
we denote the module of vector fields on $M$ and by smoothness, we will mean 
$\mathcal{C}^{\infty }$-differentiability.

\subsection{Pseudo-Finsler spaces and Berwald spaces}

A \textit{pseudo-Finsler structure (}or \textit{pseudo-Finsler metric), \cite%
{Bejancu}, }on $M$, is a smooth function $L:\mathcal{A}\rightarrow \mathbb{R}%
,$ defined on a conic subbundle\footnote{%
A conic subbundle of $TM$ is an open subset $\mathcal{A}\subset TM\setminus
0 $ with $\pi (\mathcal{A})=M$, which is stable under positive rescaling of
vectors, i.e.: $\left( x,\dot{x}\right) \in \mathcal{A}\Rightarrow \left(
x,\lambda \dot{x}\right) \in \mathcal{A},\forall \lambda >0.$} $\mathcal{A}%
\subset TM\backslash \{0\},$ with the following properties:

\begin{enumerate}
\item Positive 2-homogeneity:\ $L\left( x,\lambda \dot{x}\right) =\lambda
^{2}L\left( x,\dot{x}\right) ,$ $\forall \lambda >0.$

\item Nondegeneracy: at any $\left( x,\dot{x}\right) \in \mathcal{A}$ and in
one (then, in any) local chart around $\left( x,\dot{x}\right) $, the
Hessian:%
\begin{equation}
g_{ab}\left( x,\dot{x}\right) :=\dfrac{1}{2}\dfrac{\partial ^{2}L}{\partial 
\dot{x}^{a}\partial \dot{x}^{b}}\left( x,\dot{x}\right)
\end{equation}%
is nonsingular.
\end{enumerate}

Any pseudo-Finsler function $L$ can be continuously prolonged as $0$ at $%
\dot{x}=0$.

The conic bundle $\mathcal{A}$ is called the set of \textit{admissible
vectors.} The functions $g_{ab}=g_{ab}\left( x,\dot{x}\right) $ are the
local components of the \textit{Finslerian metric tensor} attached to $L:$ 
\begin{equation}
g:\mathcal{A}\rightarrow T_{2}^{0}M,(x,\dot{x})\mapsto g_{(x,\dot{x}%
)}=g_{ab}dx^{a}\otimes dx^{b}.
\end{equation}

In particular, pseudo-Finsler spaces include:

- pseudo-Riemannian (quadratic in $\dot{x}$) metrics: $L\left( x,\dot{x}%
\right) =a_{ab}\left( x\right) \dot{x}^{a}\dot{x}^{b};$

- classical positive definite Finsler spaces, obtained when $\mathcal{A}%
=TM\setminus \{0\}$ and $g=g\left( x,\dot{x}\right) $ - positive definite;

- various definitions of Finsler spacetimes, obtained when $g=g\left( x,\dot{%
x}\right) $ has Lorentzian signature on an appropriate conic subset of $TM$
see, e.g., \cite{Beem}, \cite{Caponio-Stancarone}, \cite{math-foundations}, 
\cite{Javaloyes2019}, \cite{Laemmerzahl-Perlick}, \cite{Minguzzi2014}.

\textit{Terminology convention:\ }In the following, we will call for
simplicity, \textit{Finsler} (respectively, \textit{Riemannian}) spaces,
pseudo-Finsler (respectively, pseudo-Riemannian) ones - and explicitly state
positive definiteness, or smoothness on the entire $TM\backslash \{0\}$,
when the case.

\bigskip

Assume, in the following, that $(M,L)$ is a Finsler space as above. A
(parametrized) admissible curve is understood as a smooth mapping $c:\left[
a,b\right] \rightarrow M,~\tau \mapsto c(\tau )\equiv \left( x^{a}\left(
\tau \right) \right) $ such that its natural extension to $TM:$%
\begin{equation}
C:\left[ a,b\right] \rightarrow TM,~\tau \mapsto (c\left( \tau \right) ,\dot{%
c}\left( \tau \right) :=\dfrac{dc}{d\tau })
\end{equation}%
lies in $\mathcal{A}.$ For an admissible curve, the Finslerian arc length is
defined as: 
\begin{equation*}
\ell _{c}=\underset{C}{\int }ds:=\overset{b}{\underset{a}{\int }}\sqrt{%
\left\vert L\left( c\left( \tau \right) ,\dot{c}\left( \tau \right) \right)
\right\vert }d\tau .
\end{equation*}%
Geodesics of $\left( M,L\right) ,$ defined as critical curves of the
functional $c\mapsto \ell _{c},$ are given, in arc length parametrization,
by:%
\begin{equation}
\frac{d^{2}x^{a}}{ds^{2}}+2G^{a}\left( x,\frac{dx}{ds}\right) =0,
\label{eq:geodesics}
\end{equation}%
where $G^{a}=\frac{1}{4}g^{ab}\left( \dot{x}^{c}\partial _{c}\dot{\partial}%
_{b}L-\partial _{b}L\right) $ and they can be geometrically described in
several equivalent ways.\ We just list here two of these:

\begin{itemize}
\item $C$ is an integral curve of the \textit{canonical spray }(or \textit{%
geodesic vector field})\textit{\ }$S\in \mathcal{X}\left( \mathcal{A}\right)
:$%
\begin{equation}
S=\dot{x}^{a}\partial _{a}-2G^{a}\dot{\partial}_{a}.  \label{def:spray}
\end{equation}

\item $\dot{C}$ is everywhere horizontal with respect to the \textit{%
canonical nonlinear connection }$N$ on $\mathcal{A},$ locally given by: 
\begin{equation}
G_{~b}^{a}=\dot{\partial}_{b}G^{a}.  \label{def:CNLin}
\end{equation}%
A nonlinear (Ehresmann) connection $N$ on $\mathcal{A}\subset TM\backslash
\{0\}$ is defined by a smooth mapping $\mathcal{H}:\mathcal{A}\mapsto T%
\mathcal{A}$, assigning to each $\left( x,\dot{x}\right) \in \mathcal{A}$ an 
$n$-dimensional\textit{\ }vector subspace $H_{\left( x,\dot{x}\right) }%
\mathcal{A}$ of $T_{\left( x,\dot{x}\right) }\mathcal{A}$ (called the 
\textit{horizontal subspace}) supplementary to the \textit{vertical subspace 
}$V_{\left( x,\dot{x}\right) }\mathcal{A}:=\ker d\pi _{\left( x,\dot{x}%
\right) }$ spanned by $\dot{\partial}_{a}:$%
\begin{equation*}
T_{\left( x,\dot{x}\right) }\mathcal{A}=H_{\left( x,\dot{x}\right) }\mathcal{%
A}\oplus V_{\left( x,\dot{x}\right) }\mathcal{A}.
\end{equation*}%
This is equivalent to the existence on each local chart of a set of
functions $G_{~b}^{a}=G_{~b}^{a}\left( x,\dot{x}\right) $ defining the \emph{%
local adapted basis} on $T\mathcal{A}$ 
\begin{equation}
\{\delta _{a}=\partial _{a}-G^{b}{}_{a}\dot{\partial}_{b};~\dot{\partial}%
_{a}\}\,,\quad H\mathcal{A}=Span\{\delta _{a}\},~\ V\mathcal{A}=Span\{\dot{%
\partial}_{a}\}.  \label{def:adapted_basis}
\end{equation}%
Accordingly, any vector field $X\in \mathcal{X}\left( \mathcal{A}\right) $
admits a unique, coordinate-invariant splitting into a horizontal component $%
hX$ and a vertical one $vX$:%
\begin{equation}
X=hX+vX=X^{a}\delta _{a}+\tilde{X}^{a}\dot{\partial}_{a},
\end{equation}%
for some functions $X^{a}=X^{a}\left( x,\dot{x}\right) ,\tilde{X}^{a}=\tilde{%
X}^{a}\left( x,\dot{x}\right) .$
\end{itemize}

An essential property to be used in the following is that, with respect to
its own canonical nonlinear connection, $L\ $is horizontally constant; in
coordinates, this reads:%
\begin{equation}
\delta _{a}L=0,~\ \ ~\ \ \ \ a=1,...,n.  \label{eq:dL=0}
\end{equation}

\subsection{Berwald spaces}

Berwald spaces are defined as "affinely connected" pseudo-Finsler spaces, in
the sense that their arc length parametrized geodesics are autoparallel
curves of a symmetric affine connection $\nabla $ on the base manifold $M$
(say, with coefficients $\Gamma _{~ab}^{c}=\Gamma _{~ab}^{c}\left( x\right) $%
):%
\begin{equation}
\frac{d^{2}x^{a}}{ds^{2}}+\Gamma _{~ab}^{c}(x)\dfrac{dx^{a}}{ds}\dfrac{dx^{b}%
}{ds}=0.  \label{eq:Berwald_geo}
\end{equation}%
This is equivalent to saying that, in any local chart, the canonical spray
coefficients $G^{c}$ are quadratic in $\dot{x},$ $G_{~b}^{c}$ are linear and 
$\dot{\partial}_{a}G_{~b}^{c}=\Gamma _{~bc}^{a}\left( x\right) $ are
independent of $\dot{x}:$ 
\begin{equation}
2G^{c}=\Gamma _{~ab}^{c}(x)\dot{x}^{a}\dot{x}^{b}~\ \Leftrightarrow
~G_{~b}^{c}=\Gamma _{~ab}^{c}(x)\dot{x}^{a}~\ \Leftrightarrow ~\dot{\partial}%
_{a}G_{~b}^{c}=\Gamma _{~ab}^{c}\left( x\right) .
\label{eq:coeffs_Berwald_spray}
\end{equation}

Obviously, all Riemannian metrics are Berwald metrics; though, as there
exist many non-Riemannian Berwald metrics (see, e.g., \cite{Bao}), a
legitimate question arises:

\begin{center}
\textit{Given a nontrivially Finslerian (non-quadratic) Berwald metric }$L$%
\textit{\ with affine connection }$\nabla ,$ \textit{does there exist a
Riemannian metric whose Levi-Civita connection is precisely }$\nabla $%
\textit{?}
\end{center}

As already stated in the Introduction, for classical, positive definite
Finsler spaces, an affirmative answer is given by the celebrated Szabo's
Metrizability Theorem, \cite{Szabo}. Yet, in Lorentzian signature, at least
one counterexample is known, \cite{Berwald-non-metrizability}.

\bigskip

To answer the question, we will switch for the moment the standpoint and
start with a fixed, symmetric connection $\nabla $ on $M$ and look for the
possible Finsler, respectively Riemann metrics whose canonical connection is 
$\nabla .$ Thus, the connection $\nabla $ is called:

\begin{enumerate}
\item \textit{Finsler metrizable }if there exists a Finsler structure $L$ on 
$M$ whose geodesics are autoparallels of $\nabla .$ Building the associated
spray, respectively, nonlinear connection of $\nabla $ by the rule (\ref%
{eq:coeffs_Berwald_spray}), this is equivalent, \cite{MZ_ELQ}, to the
existence of a nondegenerate solution $L $ of the system:%
\begin{equation}
\delta _{a}L=0,~\ \ \mathbb{C}\left( L\right) =2L,  \label{eq:metrizability}
\end{equation}%
where $\mathbb{C}=\dot{x}^{a}\dot{\partial}_{a}$ (and the last equation (\ref%
{eq:metrizability}) expresses the 2-homogeneity condition for $L$) and $%
\delta _{a}=\partial _{a}-\Gamma _{~ab}^{c}\dot{x}^{b}\dot{\partial}_{c}$
are the elements of the corresponding horizontal adapted basis (\ref%
{def:adapted_basis}).

Note:\ If $\nabla $ is Finsler metrizable, then $L$ is always of Berwald
type.

\item \textit{Riemann metrizable, }if there exists a (pseudo-)Riemannian
metric on $M$ whose Levi-Civita connection is $\nabla .$ This is equivalent
to the existence of a quadratic in $\dot{x}$ solution $L=a_{ab}\left(
x\right) \dot{x}^{a}\dot{x}^{b}$ of (\ref{eq:metrizability}); the Riemann
metrizability of an $SO(3)$-symmetric connection was investigated, e.g., in 
\cite{tanaka}.
\end{enumerate}

The above question is thus translated into:

\begin{center}
\textit{If }$\nabla $\textit{\ allows for a non-quadratic solution of (\ref%
{eq:metrizability}), does it also allow for a quadratic one?}
\end{center}

\bigskip

Noting that (\ref{eq:metrizability}) also implies:%
\begin{equation}
\left[ \delta _{a},\delta _{b}\right] L=0,~\left[ \delta _{a},\left[ \delta
_{b},\delta _{c}\right] \right] L=0,\   \label{eq:consistency_conds}
\end{equation}%
which can be interpreted as consistency conditions for the given PDE\
system, it turns out that the answer depends on the \textit{holonomy
distribution} of the corresponding spray (\ref{def:spray}), (\ref%
{eq:coeffs_Berwald_spray})). This is defined, \cite{MZ_ELQ}, as the
distribution on $\mathcal{A}$ generated by the horizontal vector fields and
their successive Lie brackets, briefly:%
\begin{equation}
{{\mathcal{D}}_{\mathcal{H}}}:=\!Span\left\{ \delta _{a},\left[ \delta
_{a},\delta _{b}\right] ,\left[ \delta _{a},\left[ \delta _{b},\delta _{c}%
\right] \right] ,...\right\} .  \label{def:DH}
\end{equation}%
Thus, ${{\mathcal{D}}_{\mathcal{H}}}$ gives the smallest involutive
distribution containing the horizontal distribution $\mathcal{H}$: 
\begin{equation*}
{{\mathcal{D}}_{\mathcal{H}}}=\mathcal{HA}\oplus \mathcal{V}({{\mathcal{D}}_{%
\mathcal{H}}}).
\end{equation*}%
Having a look at the original system (\ref{eq:metrizability}), it follows
that only the vertical parts of the Lie brackets $\left[ \delta _{a},\delta
_{b}\right] ,$ $\left[ \delta _{a},\left[ \delta _{b},\delta _{c}\right] %
\right] ,...$ - i.e., elements, of $\mathcal{V}({{\mathcal{D}}_{\mathcal{H}}}%
)$ - actually impose extra constraints on $L;$ but, the latter can be
expressed in terms of the \textit{curvature} of the corresponding nonlinear
connection $N$.

The curvature\textit{\ }of $N$ is, by definition the mapping%
\begin{equation}
R:\mathcal{X}\left( \mathcal{A}\right) \times \mathcal{X}\left( \mathcal{A}%
\right) \rightarrow \mathcal{X}\left( \mathcal{A}\right) ,~\ \ \ \ R\left(
X,Y\right) =-v\left[ hX,hY\right]  \label{def:R_hv}
\end{equation}%
measuring the non-integrability of the horizontal distribution. In
coordinates: 
\begin{equation}
R=\frac{1}{2}R_{~bc}^{a}\dot{\partial}_{a}\otimes dx^{b}\wedge dx^{c},~\ \ \ %
\left[ \delta _{b},\delta _{c}\right] =:R_{~bc}^{a}\dot{\partial}%
_{a}=(\delta _{c}G_{~b}^{a}-\delta _{b}G_{~c}^{a})\dot{\partial}_{a};\ 
\label{def:R}
\end{equation}%
for Berwald spaces, there holds:\ $R_{~bc}^{a}=R_{d~bc}^{~a}\left( x\right) 
\dot{x}^{d},$ where $R_{d~bc}^{~a}\left( x\right) $ are the curvature
components of the affine connection $\nabla .$

From (\ref{def:R_hv}), we get, for instance, that: $\mathrm{Im}\,R\subset 
\mathcal{V}({{\mathcal{D}}_{\mathcal{H}}})$ and ${{\mathcal{D}}_{\mathcal{H}}%
}=\mathcal{HA}$ if and only if $R\equiv 0$.

\subsection{Spatially spherically symmetric Berwald metrics\label%
{sec:classif_Berwald}}

Set, for the rest of the paper, $\dim M=4;$ moreover, since we are looking
for a coordinate characterization of our metrics, we can assume with no loss
of generality that $M$ is an open subset of $\mathbb{R}^{4}\backslash \{0\},$
equipped with spherical coordinates $\left( t,r,\theta ,\varphi \right) .$
In Lorentzian signature, $t$ is physically interpreted as time and $\left(
r,\theta ,\varphi \right) ,$ as spatial coordinates; we will keep for
simplicity the terminology, though we do not fix the signature of $L.$

In the following, we will assume that $L:\mathcal{A\rightarrow }\mathbb{R}$
is \textit{spatially spherically symmetric, }that is, $SO\left( 3\right) $
is a subgroup of the isometry group of $L,$ leaving the time function $t$
invariant; consequently, the generators of $SO\left( 3\right) $ solve the
Finslerian Killing equation%
\begin{equation}
X^{C}(L)=0,
\end{equation}%
where $X^{C}$ denotes the complete (natural) lift of $X\in \mathfrak{so}%
\left( 3\right) $ to $TM,$ see, e.g., \cite{Bucataru-book}.

In naturally induced coordinates $\left( x,\dot{x}\right) :=\left(
t,r,\theta ,\varphi ,\dot{t},\dot{r},\dot{\theta},\dot{\varphi}\right) $ on $%
TM,$ any $SO\left( 3\right) $-invariant, 4-dimensional Finsler metric is
known to be, \cite{Pfeifer-Wohlfarth}, of the form%
\begin{equation*}
L(x,\dot{x})=L(t,r,\dot{t},\dot{r},w),\quad w^{2}=\dot{\theta}^{2}+\dot{\phi}%
^{2}\sin ^{2}\theta \,,
\end{equation*}

Moreover, If a Berwald-Finsler function $L$ is $SO\left( 3\right) $%
-invariant, then its canonical connection $\nabla $ is also $SO\left(
3\right) $-invariant; see, e.g., \cite{Hohmann} for more details on $%
SO\left( 3\right) $-invariant connections.

In the following, we will work in the class of \textit{nontrivially
Finslerian} (i.e., non-Riemannian\footnote{%
For the class of connections corresponding to $SO\left( 3\right) $-invariant 
\textit{Riemannian }metrics, we refer to \cite{tanaka}.}) $SO\left( 3\right) 
$-invariant Berwald metrics on the said space. A full classification thereof
- and accordingly, of their canonical affine connections $\nabla $, which we
briefly present below, was found in \cite{Voicu}.

A first remark is that $\nabla $ must satisfy:%
\begin{equation}
\left[ \delta _{t},\left[ \delta _{t},\delta _{r}\right] \right] \sim \left[
\delta _{t},\delta _{r}\right] ,~\ \ \ \left[ \delta _{r},\left[ \delta
_{t},\delta _{r}\right] \right] \sim \left[ \delta _{t},\delta _{r}\right]
,\   \label{proportionality_ttr}
\end{equation}%
where the sign $\sim $ means proportionality.

Moreover, coefficients of $\nabla $ are parametrized by 10 functions $%
k_{i}=k_{i}(t,r)$ and the curvature tensor of $\Gamma $ is described by 14
coefficients $a_{i}=a_{i}(t,r)$ - see the Appendix for their precise
definitions; a useful identity, which we mention here, is:%
\begin{equation}
a_{5}=a_{9}-a_{12}.  \label{eq:a5}
\end{equation}

Finally, depending on the possible further relations between $k_{i}$ and $%
a_{i}$, there exist five classes of non-Riemannian $SO\left( 3\right) $%
-invariant Berwald-Finsler functions, split into two major cases:

\begin{itemize}
\item[I.] \textbf{If }$k_{7},k_{8},k_{9},k_{10}$\textbf{\ are not all zero},
one may assume with no loss of generality, that $k_{10}\not=0.$ Denoting%
\begin{equation}
a:=\dfrac{k_{7}}{k_{10}},~\ \ b=\dfrac{k_{8}}{k_{10}},~\ \ c=\dfrac{%
k_{9}k_{10}-k_{7}k_{8}}{k_{10}^{2}},  \label{def:abc}
\end{equation}%
the connection $\nabla $ must satisfy:%
\begin{equation}
\begin{split}
A& :=b\left( aa_{1}+a_{2}\right) +\left( ab+c\right) \left(
aa_{3}+a_{4}\right) -a_{5}\left( 2ab+c\right) =0\,, \\
B& :=a\left( aa_{3}+a_{4}\right) -\left( aa_{1}+a_{2}\right) =0, \\
C& :=\left( ab+c\right) a_{3}+b\left( aa_{3}+a_{4}\right) +b(a_{1}-2a_{5})=0,
\\
a_{6}& =aa_{7},~a_{8}=ba_{7},~\ a_{9}=\left( ab+c\right) a_{7},~\
a_{10}=aa_{11},~a_{12}=ba_{11},~\ \ a_{13}=\left( ab+c\right) a_{11}.
\end{split}
\label{eq:ABC}
\end{equation}

Further, the quantities%
\begin{eqnarray}
D &=&aa_{3}-a_{1}+a_{5},~\ \ E=ba_{3},~\ \ F=aa_{3}-a_{1}, \\
G &=&2\left( k_{1}-k_{4}a\right) ,~\ \ \tilde{G}=G-2k_{8},~\ \ \ H=2\left(
k_{2}-k_{6}a\right) ,~\ \ \ \tilde{H}=H-2k_{9},
\end{eqnarray}%
will help us to distinguish the possible classes, as follows:
\end{itemize}

\begin{enumerate}
\item[\textbf{Class 1:}] $D\neq 0$. Then, $\left[ \delta _{t},\delta _{r}%
\right] \not=0$ and $L$ has a \textit{power law }formula:%
\begin{equation}
L=\vartheta (t,r)u^{2-2\lambda }\left( v+\rho u^{2}\right) ^{\lambda }\,,
\label{def:power_law}
\end{equation}%
where:%
\begin{equation}
u=\dot{t}-a\dot{r},\quad v=c\dot{r}^{2}+2b\dot{t}\dot{r}-w^{2}\,.
\label{def:uv}
\end{equation}%
$\lambda =\frac{F}{D}\not=1$ is a constant\footnote{%
The value $\lambda =1$ corresponds to Riemannian metrics, which thus do not
satisfy the hypothesis of being properly Finslerian.}, $\rho =\frac{E}{D}$
and $\vartheta (t,r)=\exp \left( \int \left( G-\lambda \tilde{G}\right)
dt+\left( H-\lambda \tilde{H}\right) dr\right) $ (where the integral is
taken along an arbitrary path in the $\left( t,r\right) $ plane, joining
some $\left( t_{0},r_{0}\right) $ to $\left( t,r\right) $).

\item[\textbf{Class 2:}] $D=0,E,F\neq 0$. Then $\left[ \delta _{t},\delta
_{r}\right] \not=0$ and $L$ is of \textit{exponential type:}%
\begin{equation}
L=\varphi (t,r)u^{2}\exp (\frac{v}{u^{2}}\mu )\,,
\label{def_exponential_law}
\end{equation}%
where $\mu =\frac{F}{E}$, $u,v$ are as above and $\varphi (t,r)=\exp \left(
\int \left( G+2k_{4}b\mu \right) dt+\left( H+2k_{6}b\mu \right) dr\right) $
(where the integral is taken along an arbitrary path in the $\left(
t,r\right) $ plane, joining some $\left( t_{0},r_{0}\right) $ to $\left(
t,r\right) $).

\item[\textbf{Class 3:}] $D=E=F=0.$ Then, $\left[ \delta _{t},\delta _{r}%
\right] =0$ and:%
\begin{equation}
L=u^{2}\Xi (z),\quad z=\frac{v(\dot{\tilde{t}},\dot{\tilde{r}},w)}{u(\dot{%
\tilde{t}},\dot{\tilde{r}},w)^{2}}\,,  \label{def_class_3}
\end{equation}%
where $\Xi =\Xi (z)$ is an arbitrary function of the single variable $z$ and 
$u,v$ as defined above are, up to a suitable coordinate change, independent
of the new coordinates $\tilde{t}$ and $\tilde{r}$. In this case, $b$ and $c$
cannot simultaneously vanish.
\end{enumerate}

\begin{itemize}
\item[II.] \textbf{If} $k_{7}=k_{8}=k_{9}=k_{10}=0$, there exist two classes:
\end{itemize}

\begin{enumerate}
\item[\textbf{Class 4:}] $\left[ \delta _{t},\delta _{r}\right] =0.$ Then,
up to a suitable coordinate change $\left( t,r\right) \mapsto \left( \tilde{t%
},\tilde{r}\right) $, $L$ is given by a free a function of two variables: 
\begin{equation}
L=L(\dot{t},\dot{r},w)=\dot{t}^{2}L(1,p,s)\,,\quad p=\frac{\dot{r}}{\dot{t}}%
,\quad s=\frac{w}{\dot{t}}\,.  \label{def_class_5}
\end{equation}

\item[\textbf{Class 5:}] $\left[ \delta _{t},\delta _{r}\right] \not=0$.
Then $a_{1}a_{3}-a_{2}a_{4}\not=0$ and%
\begin{equation}
L=w^{2}\xi (q)\,,\quad q=\frac{\dot{t}e^{I\left( p\right) -\varphi }}{w},~\
\ \ p=\frac{\dot{r}}{\dot{t}},  \label{def_class_4}
\end{equation}%
where $\xi =\xi (q)$ is an arbitrary function of the single variable $q;$ $%
\varphi $ is a function of $t,r$ only (see \cite{Voicu} for its precise
expression) and%
\begin{equation}
I=\int \dfrac{a_{1}+a_{2}p}{a_{2}p^{2}+\left( a_{1}-a_{4}\right) p-a_{3}}dp.
\label{I_p}
\end{equation}
\end{enumerate}

Let use examine, in the following, the Riemann metrizability of these
structures.

\section{Statement of the main result\label{sec:main}}

In the following, we first make some general remarks on 4-dimensional affine
connections; then, in the case of $SO(3)$-symmetry, we state our main result.

Fix $\dim M=4$ and assume $\nabla $ is the canonical connection of an
arbitrary, nontrivially Finslerian Berwald metric $L:\mathcal{A}\rightarrow 
\mathbb{R},$ for some conic subbundle $\mathcal{A}\subset TM\backslash
\{0\}. $ Since the consistency conditions (\ref{eq:consistency_conds}) are
first order quasilinear PDE's in $L,$ which can thus be regarded as linear
and homogeneous algebraic equations in the derivatives $\dot{\partial}_{a}L$%
, it turns out that at most three of these can be independent, in other
words,%
\begin{equation}
\dim \mathcal{V}\left( \mathcal{D}_{\mathcal{H}}\right) \leq 3,
\end{equation}%
where $\mathcal{V}\left( \mathcal{D}_{\mathcal{H}}\right) $ is the vertical
part of the holonomy distribution corresponding to the spray $S$ given by %
\eqref{Eq:Geodesic_S}.

Actually, the case when the $\dim \mathcal{V}\left( \mathcal{D}_{\mathcal{H}%
}\right) $ is maximal, completely constrains the $\dot{x}$-dependence of $L,$
as shown below.

\begin{lemma}
\label{lem:maximal_dim}If $\dim \mathcal{V}\left( \mathcal{D}_{\mathcal{H}%
}\right) =3,$ then any two Finsler functions $L,\tilde{L}:\mathcal{%
A\rightarrow }\mathbb{R}$ which locally metrize $\nabla ,$ coincide up to a
constant factor.
\end{lemma}

\begin{proof}
Fix an arbitrary symmetric affine connection $\nabla $ as above and assume
that $L,\tilde{L}:\mathcal{A}\rightarrow \mathbb{R}$ both metrize $\nabla .$
As $\dim \mathcal{V}\left( \mathcal{D}_{\mathcal{H}}\right) =3,$ it follows
that the dimension of the space of solutions of the algebraic system (\ref%
{eq:consistency_conds}) is 1, meaning that there must exist a function $%
f=f\left( x,\dot{x}\right) :\mathcal{A\rightarrow }\mathbb{R},$ such that,
in any local chart: 
\begin{equation}
\dot{\partial}_{a}L=f~\dot{\partial}_{a}\tilde{L},~\ \ ~\ \ \ \ \ \ \
a=0,...,3.
\end{equation}%
Moreover, on $\mathcal{A}\backslash \tilde{L}^{-1}\{0\},$ $f$ is smooth.
Contracting this equality with $\dot{x}^{a}$ and taking into account the
2-homogeneity of $L$ and $\tilde{L},$ this implies: $L=f\tilde{L},$ which,
differentiating again by $\dot{x}^{a}$ and taking into account the above
equality, gives: $\dot{\partial}_{a}f=0$.

On the other hand, recalling that $L$ and $\tilde{L}$ both metrize $\nabla ,$
they must both be solutions of the PDE system (\ref{eq:metrizability}),
which implies $\delta _{a}f=0.$ Together with $\dot{\partial}_{a}f=0,$ this
implies that $f$ must be a constant, as claimed.
\end{proof}

For more details on the metrizability freedom of a spray, see \cite%
{Mu-Elgendi}.

\begin{remark}
\label{remark:nec_cond}\textbf{(Necessary Riemann metrizability
conditions):\ }Assume a 4-dimensional, non-Riemannian Berwald metric $L$
admits an affinely equivalent Riemannian metric. Then:

\begin{enumerate}
\item $\dim \mathcal{V}\left( \mathcal{D}_{\mathcal{H}}\right) \leq 2;$ this
follows immediately from the above Lemma.

\item The Ricci tensor of the canonical connection $\nabla $ must be
symmetric - as this is true for the Ricci tensor of any Riemannian metric.
\end{enumerate}
\end{remark}

\bigskip

In the particular case of $SO\left( 3\right) $-symmetry, we are able to
prove our main result, stating that the above conditions are also sufficient:

\begin{theorem}
\label{thm:main} Assume $L:\mathcal{A}\rightarrow \mathbb{R}$ (where $%
\mathcal{A}\subset TM\backslash \{0\}$ is a conic subbundle) is an arbitrary
signature, spatially spherically symmetric, non-Riemannian Berwald-Finsler
function on the open subset $M\subset \mathbb{R}^{4}\backslash \{0\}.$ Then,
there exists a pseudo-Riemannian metric $\mathbf{A}:TM\rightarrow \mathbb{R}$
which is affinely equivalent to $L,$ if and only if:

\begin{enumerate}
\item Its holonomy distribution $\mathcal{D}_{\mathcal{H}}$ satisfies%
\begin{equation}
\dim \mathcal{V}\left( \mathcal{D}_{\mathcal{H}}\right) \leq 2.
\label{eq:main1}
\end{equation}

\item The curvature of its affine connection has a symmetric Ricci tensor.
\end{enumerate}
\end{theorem}

\begin{proof}
The necessity of conditions 1. and 2. follows from Remark \ref%
{remark:nec_cond}.

The proof of their sufficiency will be made in the next section and
essentially relies on a case-by-case analysis of Classes 1-5. Here is an
outline:

- We first establish in Lemma \ref{lem:symmetric_Ric} the necessary and
sufficient conditions for the symmetry of the connection Ricci tensor for
any non-Riemannian Berwald-Finsler function. It turns out that the only
component of the connection Ricci tensor which may be non-symmetric is $%
R_{tr};$ also, we prove some more proportionality relations between the Lie
brackets $\left[ \delta _{a},\delta _{b}\right] .$

- For Class 1 (non-Riemannian power law metrics) and Class 2 (exponential
law metrics), we show in Lemmas \ref{lem:power_law} and \ref{lem:expo_law}
that $\dim \mathcal{V}\left( \mathcal{D}_{\mathcal{H}}\right) =3.$ Hence,
these classes do not satisfy the hypothesis 1., which, as seen above, is a
necessary condition for the Riemann metrizability of $L$; therefore, these
Finsler metrics do not admit any affinely equivalent Riemannian metric
(either spherically symmetric or not).

- Classes 3 and 4 are shown in Lemmas \ref{lem:Class_3} and \ref{lem:Class_4}
to always satisfy both the above hypotheses and to be always Riemann
metrizable.

- For $L$ belonging to Class 5, we prove in Lemma \ref{lem:Class_5} that: $%
\dim \mathcal{V}\left( \mathcal{D}_{\mathcal{H}}\right) \leq 2$, that is,
hypothesis 1. is always satisfied, whereas hypothesis 2 is equivalent to, $%
a_{1}+a_{4}=0.$ If the latter happens, we build an affinely equivalent
metric for $L$.
\end{proof}

\section{Proof of sufficiency of conditions 1. and 2.\label{sec:proofs}}

\subsection{Preliminary results\label{sec:lemmas}}

Set, in the following $M,$ as an open subset of $\mathbb{R}^{4}\backslash
\{0\}$ equipped with spherical coordinates and denote by $\nabla ,$ an $%
SO\left( 3\right) $-symmetric, torsion-free connection on $M$.

\begin{lemma}
\label{lem:symmetric_Ric}The connection Ricci tensor of an $SO\left(
3\right) $-invariant, torsion-free connection is symmetric if and only if%
\begin{equation}
R_{rt}-R_{tr}\equiv a_{1}+a_{4}+2a_{5}=0.  \label{eq:symmetric_Ricci}
\end{equation}
\end{lemma}

\begin{proof}
Use:%
\begin{equation}
R_{b~cd}^{~a}=\dot{\partial}_{b}R_{~cd}^{a},~\ \ \ R_{bc}=R_{b~ca}^{~a}
\end{equation}%
and the expressions of the curvature components in the Appendix, to find:%
\begin{eqnarray}
R_{tr} &=&R_{t~rt}^{~t}+R_{t~rr}^{~r}+R_{t~r\theta }^{~\theta
}+R_{t~r\varphi }^{~\varphi }=-a_{1}+2a_{12}, \\
R_{rt} &=&R_{r~tt}^{~t}+R_{r~tr}^{~r}+R_{r~t\theta }^{~\theta
}+R_{r~t\varphi }^{~\varphi }=a_{4}+2a_{9}
\end{eqnarray}%
Subtracting these two equalities and using the identity $a_{5}=a_{9}-a_{12},$
we find: 
\begin{equation}
R_{rt}-R_{tr}=0\ \Longleftrightarrow ~\ ~a_{1}+a_{4}+2a_{5}=0.
\end{equation}%
Moreover, a direct computation then reveals that, all the other components
of the Ricci tensor are always symmetric, as:%
\begin{equation}
R_{t\theta }=R_{\theta t}=R_{r\theta }=R_{\theta r}=R_{t\varphi }=R_{\varphi
t}=R_{r\varphi }=R_{\varphi r}=R_{\varphi \theta }=R_{\theta \varphi }=0.
\end{equation}
\end{proof}

If, moreover, $\nabla $ is the canonical affine connection of a nontrivially
Finslerian Berwald metric, then we find another result.

\begin{lemma}
\label{lem:theta_t} Denote by $\nabla $ the canonical connection of a
nontrivially Finslerian Berwald metric. Then:%
\begin{equation}
\left[ \delta _{t},\delta _{\theta }\right] \sim \left[ \delta _{r},\delta
_{\theta }\right] ,~\ \ \left[ \delta _{t},\delta _{\varphi }\right] \sim %
\left[ \delta _{r},\delta _{\varphi }\right] .  \label{eq:prop_t_theta}
\end{equation}
\end{lemma}

\begin{proof}
If $k_{7},k_{8},k_{9},k_{10}$ are not all zero, that is, the quantities $%
a,b,c$ from (\ref{def:abc}) can be constructed. Taking into account that $%
\left[ \delta _{a},\delta _{b}\right] =R_{~ab}^{c}\dot{\partial}_{c}$ and
the expressions of $R_{~t\theta }^{t},...,R_{~r\varphi }^{\varphi }$ of the
said Lie brackets in terms of the curvature coefficients $a_{i}$ (see the
Appendix), relations (\ref{eq:prop_t_theta}) turn out to be equivalent to
the last row of (\ref{eq:ABC}).

If $k_{7}=k_{8}=k_{9}=k_{10}=0,$ then, (\ref{eq:prop_t_theta}) follow from $%
a_{6}=...=a_{13}=0.$
\end{proof}

We are now ready to investigate each of the five possible classes of
nontrivially Finslerian, $SO\left( 3\right) $-invariant Berwald structures.

\subsection{Class 1:\ Power law metrics}

For this class, we obtain:

\begin{proposition}
\label{lem:power_law}: For any non-Riemannian power law Berwald metric 
\begin{equation}  \label{Eq:power_law}
L=\vartheta (t,r)u^{2-2\lambda }\left( v+\rho u^{2}\right) ^{\lambda }\,,
\end{equation}%
there hold:

\begin{enumerate}
\item $a_{5}\not=0$ (equivalently:\ $\left[ \delta _{t},\delta _{r}\right]
\not\in Span\{\dot{\partial}_{t},\dot{\partial}_{r}\}$).

\item $\dim V\left( \mathcal{D}_{\mathcal{H}}\right) =3.$

\item There is no pseudo-Riemannian metric affinely equivalent to $L.$
\end{enumerate}
\end{proposition}

\begin{proof}
\begin{enumerate}
\item Since $L$ is by hypothesis, non-Riemannian, we have $\lambda \not=1,$
meaning that:%
\begin{equation}
\dfrac{F}{D}=\dfrac{aa_{3}-a_{1}}{aa_{3}-a_{1}+a_{5}}\not=1\Rightarrow ~\
a_{5}\not=0.
\end{equation}

\item We start by noting that $a_{5}\not=0$ implies, taking into account the
identity $a_{5}=a_{9}-a_{12},$ that at least one of the coefficients $%
a_{9},a_{12}$ is nonzero. To fix things, assume $a_{9}\not=0.$ Together with 
$a_{9}=\left( ab+c\right) a_{7},$ this implies, in particular, $a_{7}\not=0$
and $ab+c\not=0.$ Then, the vector fields%
\begin{equation}
\left[ \delta _{t},\delta _{r}\right] =R_{~tr}^{a}\dot{\partial}_{a},~\ %
\left[ \delta _{t},\delta _{\theta }\right] =R_{~t\theta }^{a}\dot{\partial}%
_{a},~\ \ \left[ \delta _{t},\delta _{\varphi }\right] =R_{~t\varphi }^{a}%
\dot{\partial}_{a}
\end{equation}%
are linearly independent. This can be easily seen by checking the rank of
the matrix consisting of their components, that is:%
\begin{equation}
\left( 
\begin{array}{cccc}
a_{1}\dot{t}+a_{2}\dot{r} & a_{3}\dot{t}+a_{4}\dot{r} & a_{5}\dot{\theta} & 
a_{5}\dot{\varphi} \\ 
a_{6}\dot{\theta} & a_{7}\dot{\theta} & a_{8}\dot{t}+a_{9}\dot{r} & 0 \\ 
a_{6}\dot{\varphi}\sin ^{2}\theta & a_{7}\dot{\varphi}\sin ^{2}\theta & 0 & 
a_{8}\dot{t}+a_{9}\dot{r}%
\end{array}%
\right) .
\end{equation}%
We find that the determinant obtained using the last three columns:%
\begin{equation}
\Delta :=(a_{8}\dot{t}+a_{9}\dot{r})(a_{3}a_{8}\dot{t}%
^{2}+(a_{3}a_{9}+a_{4}a_{8})\dot{t}\dot{r}+a_{4}a_{9}\dot{r}%
^{2}-a_{5}a_{7}w^{2})
\end{equation}%
is nonzero since $a_{5}$, $a_{7}$, and $a_{9}$ are nonzero. This ensures
that $\dim V\left( \mathcal{D}_{\mathcal{H}}\right) =3.$

\item The statement now follows immediately from $\dim V\left( \mathcal{D}_{%
\mathcal{H}}\right) =3$ and Remark \ref{remark:nec_cond}.
\end{enumerate}
\end{proof}

\begin{remark}
All power law metrics with $\lambda \neq 0,1$ are counterexamples to Sz\'{a}b%
\'{o}'s Metrizability Theorem \cite{Szabo} on non-positive definite Berwald
metrics.
\end{remark}

Below, we present two concrete examples of connections metrizable by
non-Riemannian, Berwald-type power law metrics; in the first case, the
connection Ricci tensor is not symmetric, whereas in the second one, it is
symmetric.

\begin{example}[Power law, non-symmetric Ricci tensor]
Let the functions $k_{i}$ be given by:%
\begin{equation}
k_{1}=2r(\alpha -2),~\ k_{4}=4\alpha r^{3}(\alpha -1),~\ k_{6}=-2\alpha r,\
~k_{8}=-2r,\ k_{10}=\alpha r,
\end{equation}%
where $\alpha $ is an arbitrary constant, and $%
k_{2}=k_{3}=k_{5}=k_{7}=k_{9}=0.$ Accordingly, 
\begin{equation}
a=0,~\ b=-\frac{2}{\alpha },~\ c=0.
\end{equation}%
The corresponding connection $\nabla $ has the nonzero curvature components: 
\begin{eqnarray}
a_{1} &=&2\alpha -4,~\ a_{3}=4\alpha r^{2}\left( \alpha -1\right) ,~\
a_{4}=-2\alpha ,~\ a_{5}=-2, \\
a_{7} &=&2\alpha r^{2}\left( \alpha -1\right) ,~\ \ a_{8}=-4r^{2}\left(
\alpha -1\right) ,~\ \ a_{11}=-\alpha ,~\ a_{12}=2,~\ \ a_{14}=1.
\end{eqnarray}%
These satisfy the constraints (\ref{proportionality_ttr}) and (\ref{eq:ABC}%
), meaning that $\nabla $ is Finsler metrizable. Moreover:%
\begin{equation}
R_{rt}-R_{tr}\equiv a_{1}+a_{4}+2a_{5}=-8\not=0,
\end{equation}%
meaning that the connection Ricci tensor is not symmetric.

To construct the Finsler function, we calculate:%
\begin{eqnarray}
D &=&2-2\alpha ,~\ \ E=-8(\alpha -1)r^{2},~F=4-2\alpha ,~\ \ \lambda =\frac{%
2-\alpha }{1-\alpha },\quad \rho =4r^{2}. \\
G &=&4r(\alpha -2),\ \overline{G}=4r(\alpha -1),\ H=0,~\ \overline{H}=0.
\end{eqnarray}%
leading to: 
\begin{equation*}
L=\dot{t}^{\frac{2}{\alpha -1}}(4\alpha \ r^{2}\dot{t^{2}}-4\dot{t}\dot{r}%
-\alpha \ \dot{\theta}^{2}-\alpha \dot{\phi}^{2}\sin ^{2}\theta )^{\frac{%
\alpha -2}{\alpha -1}}.
\end{equation*}%
For $\alpha \not=1,$ this is a well defined Finsler function on the conic
domain $\mathcal{A}$ given by: $\dot{t}>0,$ $4\alpha \ r^{2}\dot{t^{2}}-4%
\dot{t}\dot{r}-\alpha \ \dot{\theta}^{2}-\alpha \sin ^{2}\theta >0.$ It can
be checked to be of Lorentzian signature, if $\alpha >2.$
\end{example}

\begin{example}[Power law, symmetric Ricci tensor]
Consider an arbitrary smooth function $\Phi =\Phi \left( t,r\right) $
without zeros and build the connection $\nabla $ by:%
\begin{equation}
k_{1}=\partial _{t}\Phi ,~\ \ k_{5}=k_{9}=k_{10}=\dfrac{1}{3}\partial
_{r}\Phi ,~\ \ \ k_{2}=k_{3}=k_{4}=k_{6}=k_{7}=k_{8}=0;
\end{equation}%
this gives: $a=b=0,~c=1$ and the nonzero curvature components:%
\begin{equation}
a_{1}=\partial _{t}\partial _{r}\Phi ,~\ \
a_{4}=a_{5}=a_{7}=a_{9}=a_{11}=a_{13}=-\dfrac{1}{3}\partial _{t}\partial
_{r}\Phi ,~\ a_{14}=1+\dfrac{1}{9}\left( \partial _{t}\partial _{r}\Phi
\right) ^{2}.
\end{equation}%
In particular, 
\begin{equation}
R_{rt}-R_{tr}=a_{1}+a_{4}+2a_{5}=0,
\end{equation}%
meaning that the connection Ricci tensor is symmetric. Moreover, the
curvature constraints (\ref{proportionality_ttr}) and (\ref{eq:ABC}) are
identically satisfied, hence the connection is Finsler metrizable. To
construct the Finsler function, we build:%
\begin{eqnarray}
D &=&-\dfrac{4}{3}\partial _{t}\partial _{r}\Phi ,~\ E=0,~\ \ F=-\partial
_{t}\partial _{r}\Phi \ \Rightarrow \ \ \lambda =\dfrac{3}{4},~\ \rho =0, \\
G &=&\tilde{G}=2\partial _{t}\Phi ,~H=0,~\ \tilde{H}=-\dfrac{2}{3}\partial
_{r}\Phi \ \Rightarrow \vartheta =\dfrac{1}{2}\Phi ,
\end{eqnarray}%
which leads to:%
\begin{equation}
L:=e^{\tfrac{1}{2}\Phi \left( t,r\right) }\dot{t}^{1/2}\left( \dot{r}^{2}-%
\dot{\theta}^{2}-\dot{\phi}^{2}\sin ^{2}\theta \right) ^{3/4}.
\end{equation}%
The obtained function $L$ is smooth on the conic subbundle $\mathcal{A}:=\{%
\dot{t}>0,~\dot{r}^{2}-\dot{\theta}^{2}-\dot{\varphi}^{2}\sin ^{2}\theta \};$
moreover, the inequality $\det g=-\dfrac{27}{16}e^{2\Phi }\sin ^{2}\theta <0$
tells us that $L$ is of Lorentzian signature.
\end{example}

\subsection{Class 2:\ exponential metrics}

A quick look at the defining formula of this class%
\begin{equation}
L=\varphi (t,r)u^{2}\exp \left( \frac{v}{u^{2}}\mu \right) \,,~\ \ \mu =%
\dfrac{F}{E}  \label{Eq:exponential_law}
\end{equation}%
shows that this contains no Riemannian metric.

Let us check, in this case, $\dim V\left( \mathcal{D}_{\mathcal{H}}\right) .$
A first remark is that $F=aa_{3}-a_{1}$ cannot vanish, as it would lead to a
degenerate $L.$ Together with $D=F+a_{5}=0,$ we find, again, that 
\begin{equation}
a_{5}\not=0.
\end{equation}%
By a similar reason to points 2 and 3 above, we find, again, that $\left[
\delta _{t},\delta _{r}\right] ,\left[ \delta _{t},\delta _{\theta }\right] ,%
\left[ \delta _{t},\delta _{\varphi }\right] $ are linearly independent
generators of $V\left( \mathcal{D}_{\mathcal{H}}\right) .$ We have thus
found:

\begin{lemma}
\label{lem:expo_law} Berwald-type, exponential law metrics are never
affinely equivalent to any Riemannian metric. They always obey:%
\begin{equation}
\dim V\left( \mathcal{D}_{\mathcal{H}}\right) =3.
\end{equation}
That is, they are counterexamples to Sz\'{a}b\'{o}'s Metrizability Theorem 
\cite{Szabo} in the case of non-positive definiteness.
\end{lemma}

We have the following concrete example of an exponential law metric.

\begin{example}[Exp. case]
Let the functions $k_{i}$ be given by%
\begin{eqnarray}
k_{1} &=&r-4t-re^{(r-t)^{2}}+3t^{3}-5rt^{2}+2r^{2}t,~\
k_{2}=re^{(r-t)^{2}}-3t^{3}+5rt^{2}-2r^{2}t+2, \\
k_{3} &=&-(k_{1}+2k_{2}),~k_{4}=2k_{1}+k_{2}+2t,~\ k_{5}=-k_{2}+2t,~\
k_{6}=-k_{1}-2t, \\
k_{7} &=&k_{9}=k_{10}=t,~\ k_{8}=-t,
\end{eqnarray}%
which gives:%
\begin{equation}
a=1,~\ b=-1,~\ c=2.
\end{equation}%
A direct calculation of the curvature components $a_{i}$ then shows that the
connection satisfies the Finsler metrizability conditions (\ref%
{proportionality_ttr}) and (\ref{eq:ABC}). Moreover,%
\begin{equation}
D=0,~\ E=e^{(r-t)^{2}},~F=1,~\ \mu =e^{-(r-t)^{2}},\quad \varphi (t,r)=\exp
((3t^{2}-2rt-1)e^{-(r-t)^{2}}),
\end{equation}%
leading to:%
\begin{equation*}
L=\exp ((3t^{2}-2rt-1)e^{-(r-t)^{2}})\exp {\large (}\frac{e^{-(r-t)^{2}}(2%
\dot{r}^{2}-2\dot{t}\dot{r}-\dot{\theta}^{2}-\dot{\phi}^{2}\sin ^{2}\theta )%
}{(\dot{r}-\dot{t})^{2}}{\large )}(\dot{r}-\dot{t})^{2}.
\end{equation*}%
This is a well defined, smooth (pseudo-)Finsler function on the conic
subbundle $\mathcal{A}\subset TM\backslash \{0\}$ defined by $\dot{r}-\dot{t}%
\not=0.$
\end{example}

\subsection{Class 3: $D=E=F=0$}

The defining conditions of this class actually imply: $%
a_{1}=a_{2}=a_{3}=a_{4}=a_{5}=0,$ which is equivalent to%
\begin{equation}
\lbrack \delta _{t},\delta _{r}]=0.
\end{equation}

In this case, we will explicitly solve the Berwald conditions $\delta
_{a}L=0 $ to prove:

\begin{lemma}
\label{lem:Class_3} Assume $L$ belongs to Class 3. Then:

\begin{enumerate}
\item $L$ must be of the form:%
\begin{equation}
L=(e^{\mathcal{G}}u^{2})\Theta \left( ze^{-\left( \mathcal{G}-2\mathcal{K}%
\right) }+\mathcal{M}\right) ,  \label{general_sol_Berwald_class_3}
\end{equation}%
where $z:=\dfrac{v}{u^{2}},$ $\Theta $ is a free function of a single
variable and $\mathcal{G},\mathcal{K}$ and $\mathcal{M}$ are functions of $t$
and $r$ given by:%
\begin{eqnarray}
2\left( k_{1}-k_{4}a\right) &=&\partial _{t}\mathcal{G},~\ 2\left(
k_{2}-k_{6}a\right) =\partial _{r}\mathcal{G},~\ \ k_{8}=\partial _{t}%
\mathcal{K},~\ \ k_{9}=\partial _{r}\mathcal{K},  \label{def_G_K} \\
e^{-\left( \mathcal{G}-2\mathcal{K}\right) }bk_{6} &=&\dfrac{1}{2}\partial
_{r}\mathcal{M},~\ \ \ e^{-\left( \mathcal{G}-2\mathcal{K}\right) }bk_{4}=%
\dfrac{1}{2}\partial _{t}\mathcal{M}.  \label{def_M}
\end{eqnarray}

\item $L$ is affinely equivalent to the spatially spherically symmetric
pseudo-Riemannian metric%
\begin{equation}
\mathbf{A}=ve^{2\mathcal{K}}+\left( e^{\mathcal{G}}\mathcal{M}\right) u^{2}.
\end{equation}

\item $\dim V\left( \mathcal{D}_{\mathcal{H}}\right) \leq 2$ and the Ricci
tensor of the canonical affine connection of $L$ is symmetric.
\end{enumerate}
\end{lemma}

\begin{proof}
\begin{enumerate}
\item Let us show that the functions $\mathcal{G},\mathcal{K}$ and $\mathcal{%
M}$ in (\ref{def_G_K})-(\ref{def_M}) really exist. To this aim, we use the
identities, \cite{Voicu}: 
\begin{equation}
F=\dfrac{1}{2}\left( \partial _{t}H-\partial _{r}G\right) ,~\ D=\dfrac{1}{2}%
\left( \partial _{t}\tilde{H}-\partial _{r}\tilde{G}\right) ,~\ E=b\left(
k_{6}\tilde{G}-k_{4}\tilde{H}\right) +\partial _{r}\left( k_{4}b\right)
-\partial _{t}\left( k_{6}b\right) .
\end{equation}%
From $F=0,$ we find that there must exist a function $\mathcal{G=G}\left(
t,r\right) $ such that%
\begin{equation}
G=\partial _{t}\mathcal{G},~\ H=\partial _{r}\mathcal{G}.  \label{GH}
\end{equation}%
Further, $D\equiv F+\left( k_{8,r}-k_{9,t}\right) =0$ gives: $%
k_{8,r}=k_{9,t},$ which is just the integrability condition for the system%
\begin{equation}
k_{8}=\partial _{t}\mathcal{K},~\ \ k_{9}=\partial _{r}\mathcal{K}.
\end{equation}%
Finally, $E=0,$ after a multiplication by $e^{-\left( \mathcal{G}-2\mathcal{K%
}\right) }$, becomes the integrability condition for (\ref{def_M}):%
\begin{equation}
\partial _{t}\left( e^{-\left( \mathcal{G}-2\mathcal{K}\right)
}bk_{6}\right) =\partial _{r}\left( e^{-\left( \mathcal{G}-2\mathcal{K}%
\right) }bk_{4}\right) .
\end{equation}%
We are now able to completely integrate the system $\delta _{a}L=0;$ as
shown in \cite{Voicu}, any function $L$ as in \ref{def_class_3} solves the
equations $\delta _{\theta }L=0,~\delta _{\phi }L=0$; moreover, the
remaining equations $\delta _{t}L=0,~\delta _{r}L=0$ are equivalent to:%
\begin{equation}
\left\{ 
\begin{array}{c}
\partial _{t}\Xi =\Xi G-\partial _{z}\Xi \left( \tilde{G}z-2k_{4}b\right) \\ 
\partial _{r}\Xi =\Xi H-\partial _{z}\Xi \left( \tilde{H}z-2k_{6}b\right)%
\end{array}%
\right. ,
\end{equation}%
which, taking into account the above relations, read:%
\begin{equation}
\left\{ 
\begin{array}{c}
\partial _{t}\Xi =\Xi \partial _{t}\mathcal{G}-\partial _{z}\Xi \left(
z\partial _{t}\mathcal{\tilde{G}}-e^{\mathcal{\tilde{G}}}\partial _{t}%
\mathcal{M}\right) \\ 
\partial _{r}\Xi =\Xi \partial _{r}\mathcal{G}-\partial _{z}\Xi \left(
z\partial _{r}\mathcal{\tilde{G}}-e^{\mathcal{\tilde{G}}}\partial _{r}%
\mathcal{M}\right)%
\end{array}%
\right. ,  \label{eqs_Csi_1}
\end{equation}%
where $\mathcal{\tilde{G}}:=\left( \mathcal{G}-2\mathcal{K}\right) .$ With
the substitution $\Xi =e^{\mathcal{G}}\Theta ,$ these become%
\begin{equation}
\left\{ 
\begin{array}{c}
\partial _{t}\Theta +\partial _{z}\Theta \left( z\partial _{t}\mathcal{%
\tilde{G}}-e^{\mathcal{\tilde{G}}}\partial _{t}\mathcal{M}\right) =0 \\ 
\partial _{r}\Theta +\partial _{z}\Theta \left( z\partial _{r}\mathcal{%
\tilde{G}}-e^{\mathcal{\tilde{G}}}\partial _{r}\mathcal{M}\right) =0%
\end{array}%
\right.
\end{equation}%
and are easily shown to admit the general solution (\ref%
{general_sol_Berwald_class_3}).

\item A quadratic solution is obtained for $\Theta :=id.$:%
\begin{equation}
\mathbf{A}=ve^{2\mathcal{K}}+\left( e^{\mathcal{G}}\mathcal{M}\right) u^{2}.
\label{A}
\end{equation}%
To prove that $\mathbf{A}$ is nondegenerate, we note that the Hessian $%
\left( a_{ab}\right) $ of $\mathbf{A}$ has the determinant: 
\begin{equation}
\det \left( a_{ab}\right) =\sin ^{2}\theta e^{6\mathcal{K}}\Delta ,~\ \ \
\Delta :=\mathcal{M}e^{\mathcal{G}}\left( 2ab+c\right) -b^{2}e^{2\mathcal{K}%
}.
\end{equation}%
Since $\mathcal{M}$ depends on a free integration constant, and $2ab+c$ and $%
b$ cannot simultaneously vanish (as $b$ and $c$ cannot simultaneously
vanish), one can choose $\mathcal{M}$ such that $\Delta \not=0,$ at least
for $(t,r)$ in a specified domain.

\item Assume $L$ is non-Riemannian, that is, $\Theta $ is non-constant.
Then, $L$ and the quadratic function $\mathbf{A}$ in $\dot{x}$ given by %
\eqref{A} share the same geodesic spray \eqref{Eq:Geodesic_S}. By Remark \ref%
{remark:nec_cond}, it follows that $\dim V\left( \mathcal{D}_{\mathcal{H}%
}\right) \leq 2.$ The symmetry of the connection Ricci tensor is a direct
consequence of $a_{1}=...=a_{5}=0$ and of Lemma \ref{lem:symmetric_Ric}.
\end{enumerate}
\end{proof}

\textbf{Remark: }Cosmologically symmetric Berwald metrics (see Appendix B of 
\cite{Voicu}) and flat Berwald metrics $a_{i}=0,\forall i=1,...,14,$ belong
to this class. The latter is seen as $a_{14}=1+k_{7}k_{8}+k_{9}k_{10}=0$
prevents $k_{7},k_{8},k_{9},k_{10}$ from vanishing, whereas, at the same
time, $\left[ \delta _{t},\delta _{r}\right] =0.$

\subsection{Class 4:\ metrics with $\left[ \protect\delta _{t},\protect%
\delta _{r}\right] =0,$ $k_{7}=k_{8}=k_{9}=k_{10}=0$}

For connections corresponding to this class, we find:%
\begin{equation}
a_{1}=....=a_{13}=0,~\ a_{14}=1,  \label{eq:zero_a6a13}
\end{equation}%
that is, on the one hand, the connection Ricci tensor is symmetric and, on
the other hand:%
\begin{equation}
\left[ \delta _{t},\delta _{r}\right] =\left[ \delta _{t},\delta _{\theta }%
\right] =\left[ \delta _{t},\delta _{\varphi }\right] =\left[ \delta
_{r},\delta _{\theta }\right] =\left[ \delta _{r},\delta _{\varphi }\right]
=0.
\end{equation}%
The successive Lie brackets $[[\delta _{a},\delta _{b}],\delta _{c}]$ do not
generate new vertical directions since:%
\begin{equation*}
\lbrack \delta _{\theta },[\delta _{\theta },\delta _{\varphi }]]=0,\quad
\lbrack \delta _{\varphi },[\delta _{\theta },\delta _{\varphi }]]=-\dot{%
\varphi}\cos \theta \sin \theta \dot{\partial}_{\theta }+\frac{\cos \theta }{%
\sin \theta }\dot{\theta}\dot{\partial}_{\varphi }=\cot \theta \ [\delta
_{\theta },\delta _{\varphi }].
\end{equation*}%
Therefore, $\mathcal{V}\left( \mathcal{D}_{\mathcal{H}}\right) $ is
generated by $\left[ \delta _{\theta },\delta _{\varphi }\right] $ alone,
meaning that $\dim \mathcal{V}\left( \mathcal{D}_{\mathcal{H}}\right) =1.$

\bigskip

To prove that these connections are always Riemann metrizable, we note that
the defining relation $\left[ \delta _{t},\delta _{r}\right] =0$ means that,
actually, the "$tr$-corner"\ $k_{1},...k_{6}$ of the connection is a flat
2-dimensional connection in its own right. That is, this is always
metrizable by a flat (pseudo-Euclidean) 2-dimensional metric $h=h\left(
t,r\right) $ of the desired signature. A quadratic metric\textbf{\ }$\mathbf{%
A}$ on $M$ is then built as:%
\begin{equation}
\mathbf{A:}=h_{tt}\dot{t}^{2}+2h_{tr}\dot{t}\dot{r}+h_{rr}\dot{r}^{2}\pm (%
\dot{\theta}^{2}+\dot{\phi}^{2}\sin ^{2}\theta ).  \label{class_4_Riemann}
\end{equation}%
The Levi-Civita connection of $\mathbf{A}$ then has the precise coefficient
components $k_{1},...k_{6},k_{7}=k_{8}=k_{9}=k_{10}=0.$ We have thus
obtained:

\begin{lemma}
\label{lem:Class_4} For any spatially symmetric Berwald metric whose
canonical connection obeys $\left[ \delta _{t},\delta _{r}\right] =0,$ $%
k_{7}=k_{8}=k_{9}=k_{10}=0,$ there exists an affinely equivalent Riemannian
metric $\mathbf{A},$ built as in (\ref{class_4_Riemann}). In this case:%
\begin{equation}
\dim \mathcal{V}\left( \mathcal{D}_{\mathcal{H}}\right) =1.
\end{equation}
\end{lemma}

\bigskip

\textbf{A concrete example. } Take $k_{1}=...=k_{10}=0$ in the given chart
and $h_{tt}=1,$ $h_{rr}=-1,$ $h_{tr}=0.$ Then, the Lorentzian metric%
\begin{equation}
\mathbf{A}:=\dot{t}^{2}-\dot{r}^{2}-w^{2}
\end{equation}%
has its Levi-Civita connection $\tilde{\Gamma}$ with the only nonzero
coefficients:%
\begin{equation}
\tilde{\Gamma}_{~\phi \phi }^{\theta }=-\sin \theta \cos \theta ,~\ \ \tilde{%
\Gamma}_{~\theta \phi }^{\phi }=\cot \theta ,
\end{equation}%
meaning that it metrizes the given connection.

\subsection{Class 5:\ $\left[ \protect\delta _{t},\protect\delta _{r}\right]
\not=0,$ $k_{7}=k_{8}=k_{9}=k_{10}=0$}

\begin{lemma}
\label{lem:Class_5} Assume $L$ is a nontrivially Finslerian metric belonging
to Class 5; in particular: $a_{1}a_{4}-a_{2}a_{3}\not=0.$ Then:

\begin{enumerate}
\item There hold the relations:%
\begin{equation}
a_{5}=0,~\ \ R_{rt}-R_{tr}=a_{1}+a_{4}.
\end{equation}%
In particular, the symmetry of the connection Ricci tensor is equivalent to $%
a_{1}+a_{4}=0.$

\item If $a_{1}+a_{4}=0,$ then $L$ is affinely equivalent to any of the
pseudo-Riemannian metrics:%
\begin{equation}
\mathbf{A}=\mathcal{C}_{1}e^{-2\varphi }\left\vert -a_{3}\dot{t}^{2}+2a_{1}%
\dot{t}\dot{r}+a_{2}\dot{r}^{2}\right\vert +\mathcal{C}_{2}w^{2},
\label{A_Class_5}
\end{equation}%
where $\mathcal{C}_{1},\mathcal{C}_{2}\in \mathbb{R}\backslash \{0\}$ are
constants.

\item The canonical connection of $L$ always has $\dim \mathcal{V}\left( 
\mathcal{D}_{\mathcal{H}}\right) \leq 2.$
\end{enumerate}
\end{lemma}

\begin{proof}
\begin{enumerate}
\item The statement follows immediately from $a_{5}=k_{8,r}-k_{9,t}$ and
from Lemma \ref{lem:symmetric_Ric}.

\item Finsler metrics in this class are given, see Section \ref%
{sec:classif_Berwald}, by:%
\begin{equation}
L=w^{2}\xi \left( q\right) ,~\ \ q=\dfrac{\dot{t}}{w}e^{I\left( t,r,p\right)
-\varphi },~\ p=\frac{\dot{r}}{\dot{t}},\   \label{L_class_5}
\end{equation}%
where $\xi $ is a free function and $\varphi =\varphi \left( t,r\right) $.
Let us find a concrete $\xi $ which leads to quadratic functions in $\dot{t},%
\dot{r},\dot{\theta},\dot{\phi}.$

Since $w=(\dot{\theta}^{2}+\dot{\phi}^{2}\sin ^{2}\theta )^{1/2}$ is given
by a square root and $\dot{t},p,\varphi $ are independent of $w$, the only
chance to obtain a quadratic function in $\dot{\theta},\dot{\phi}$ is to
pick:%
\begin{equation}
\xi \left( q\right) =\mathcal{C}_{1}q^{2}+\mathcal{C}_{2},
\end{equation}%
where $\mathcal{C}_{1}$ and $\mathcal{C}_{2}$ are constants, that is: 
\begin{equation}
\mathbf{A}=\mathcal{C}_{1}e^{-2\varphi }\dot{t}^{2}e^{2I\left( t,r,p\right)
}+\mathcal{C}_{2}w^{2}.  \label{eq:A_rough_Class5}
\end{equation}%
The above expression $\mathbf{A}$ is also quadratic in $\dot{t}$ and $\dot{r}
$ if and only if $e^{2I\left( t,r,p\right) }$ is the ratio between a
quadratic expression in $\dot{t},\dot{r}$ and $\dot{t}^{2}:$%
\begin{equation}
e^{2I\left( t,r,p\right) }=\dfrac{\alpha \dot{t}^{2}+\beta \dot{t}\dot{r}%
+\gamma \dot{r}^{2}}{\dot{t}^{2}}=\alpha +\beta p+\gamma p^{2}
\end{equation}%
for some functions $\alpha ,\beta $ and $\gamma $ of $\left( t,r\right) ,$
Taking a logarithm in the above, we get:%
\begin{equation}
I=\dfrac{1}{2}\ln \left\vert \alpha +\beta p+\gamma p^{2}\right\vert ,
\end{equation}%
which, by virtue of (\ref{I_p}) and of the hypothesis $a_{4}=-a_{1}$, gives:%
\begin{equation}
\dfrac{\partial I}{\partial p}=\dfrac{1}{2}\dfrac{2\gamma p+\beta }{\alpha
+\beta p+\gamma p^{2}}=\dfrac{a_{2}p+a_{1}}{a_{2}p^{2}+2a_{1}p-a_{3}}.
\end{equation}%
This is identically solved by: $\alpha =-\dfrac{a_{3}}{2},~\ \ \beta
=a_{1},~\ \ \gamma =\dfrac{a_{2}}{2},$ that is:%
\begin{equation}
I=\dfrac{1}{2}\ln \left\vert -a_{3}+2a_{1}p+a_{2}p^{2}\right\vert ,
\end{equation}%
which, substituted into (\ref{eq:A_rough_Class5}), leads to:%
\begin{equation}
\mathbf{A}=\mathcal{C}_{1}e^{-2\varphi }\left\vert -a_{3}\dot{t}^{2}+2a_{1}%
\dot{t}\dot{r}+a_{2}\dot{r}^{2}\right\vert +\mathcal{C}_{2}w^{2}
\end{equation}%
metrizes the given connection. A quick check shows that its associated
metric tensor has the determinant: 
\begin{equation*}
\mathcal{C}_{1}^{2}\mathcal{C}_{2}^{2}e^{-4\varphi
}(a_{1}^{2}+a_{2}a_{3})\sin ^{2}\theta .
\end{equation*}%
The hypothesis $0\not=a_{1}a_{4}-a_{2}a_{3}=-a_{1}^{2}-a_{2}a_{3}$ ensures
that the bracket above is nonzero, that is, $\mathbf{A}$ is a well defined
Riemannian metric for any $\mathcal{C}_{1}\neq 0,\quad \mathcal{C}%
_{2}\not=0. $

\item The statement follows again, from Remark \ref{remark:nec_cond}.
\end{enumerate}
\end{proof}

We note that $k_{7}=k_{8}=k_{9}=k_{10}$ implies: $a_{6}=...=a_{13}=0,$ $%
a_{14}=1,$ that is:%
\begin{equation}
\left[ \delta _{t},\delta _{\theta }\right] =\left[ \delta _{r},\delta
_{\theta }\right] =\left[ \delta _{t},\delta _{\varphi }\right] =\left[
\delta _{r},\delta _{\varphi }\right] =0.
\end{equation}%
This means that $\mathcal{V}\left( \mathcal{D}_{\mathcal{H}}\right) $ is, at
each point, generated by $\left[ \delta _{t},\delta _{r}\right] $ and $\left[
\delta _{\theta },\delta _{\phi }\right] ,$ which are easily seen to be
linearly independent and to commute. Another interesting remark is that, for 
$a_{1}+a_{4}\not=0$, the respective Finsler metrics have a non-symmetric
connection Ricci tensor, though $\dim \mathcal{V}\left( \mathcal{D}_{%
\mathcal{H}}\right) <3$.

\section{Conclusion}

In the present paper, we have answered the question on which $SO\left(
3\right) $-symmetric, nontrivially Finslerian Berwald 4-dimensional
structures admit/do not admit, an affinely equivalent pseudo-Riemannian one.
In the affirmative case, we have explicitly constructed the respective
pseudo-Riemannian metric. The situation is summarized in the table below,
where $\delta _{w}=0$ (respectively, $\delta _{w}\not=0$) is a shorthand
writing for: $k_{7}=k_{8}=k_{9}=k_{10}=0$ (respectively, not all of $%
k_{7},k_{8},k_{9},k_{10}$ are zero):

\begin{center}
\begin{tabular}{|l|l|l|l|}
\hline
\textbf{Cases} & \textbf{Defining properties} & $\dim \mathcal{V}\left( 
\mathcal{D}_{\mathcal{H}}\right) $ & \textbf{Riemann equiv.} \\ \hline
Class 1: Power law & $\left[ \delta _{t},\delta _{r}\right] \not=0,$ $%
D\not=0,~\delta _{w}\not=0$ & $3$ & no \\ \hline
Class 2: Exp. law & $\left[ \delta _{t},\delta _{r}\right] \not=0,$ $%
D=0,~E\not=0,\delta _{w}\not=0$ & $3$ & no \\ \hline
Class 3: & $\left[ \delta _{t},\delta _{r}\right] =0,~\ \delta _{w}\not=0$ & 
$\leq 2$ & yes \\ \hline
Class 4: & $\left[ \delta _{t},\delta _{r}\right] =0,~\ \delta _{w}=0$ & $1$
& yes \\ \hline
Class 5: & $\left[ \delta _{t},\delta _{r}\right] \not=0,~\ \delta _{w}=0$ & 
$\leq 2$ & iff Ricci=symmetric \\ \hline
\end{tabular}
\end{center}

The question is relevant for physics in Lorentzian signature, where $%
SO\left( 3\right) $-symmetric Berwald metrics are candidates for models of
the gravitational field around massive bodies such as neutron stars or black
holes. Berwald-Finsler spacetimes that are \textit{not} affinely Riemannian
will produce geodesics and geodesic deviations that cannot be ascribed to
any Riemannian metric, thus, predicting specific Finslerian effects - which,
in principle, should be detectable. The latter will be used in a future
paper, as ansatzes in a Finslerian generalization of the Einstein field
equation, \cite{kinetic gas}, \cite{math-foundations}, to determine the
gravitational field of a spherically symmetric kinetic gas.

Work in progress is the determination of necessary and sufficient conditions
for the existence of an affinely equivalent pseudo-Riemannian metric, for
general Berwald-type pseudo-Finsler spaces.

\section{Appendix: Connection coefficients and curvature components}

\label{app:defs} The most general $SO(3)$-invariant, symmetric affine
connection is given, in 4 dimensions, by the coefficients, \cite{Voicu}:%
\begin{align}
\Gamma _{tt}^{t}& =k_{1}(t,r), & \Gamma _{tr}^{t}& =k_{2}(t,r),  \notag \\
\Gamma _{rr}^{t}& =k_{3}(t,r), & \Gamma _{tt}^{r}& =k_{4}(t,r),  \notag \\
\Gamma _{rr}^{r}& =k_{5}(t,r), & \Gamma _{tr}^{r}& =k_{6}(t,r),  \notag \\
\Gamma _{\theta \theta }^{t}& =\tfrac{\Gamma _{\phi \phi }^{t}}{\sin
^{2}\theta }=k_{7}(t,r), & \Gamma _{\phi t}^{\phi }& =\Gamma _{\theta
t}^{\theta }=k_{8}(t,r),  \notag \\
\Gamma _{\phi r}^{\phi }& =\Gamma _{\theta r}^{\theta }=k_{9}(t,r), & \Gamma
_{\theta \theta }^{r}& =\tfrac{\Gamma _{\phi \phi }^{r}}{\sin ^{2}\theta }%
=k_{10}(t,r),  \notag \\
\sin \theta \Gamma _{t\theta }^{\phi }& =-\tfrac{\Gamma _{\phi t}^{\theta }}{%
\sin \theta }=k_{11}(t,r), & \Gamma _{\phi \phi }^{\theta }& =-\sin \theta
\cos \theta  \notag \\
\sin \theta \Gamma _{r\theta }^{\phi }& =-\tfrac{\Gamma _{r\phi }^{\theta }}{%
\sin \theta }=k_{12}(t,r), & \Gamma _{\theta \phi }^{\phi }& =\Gamma _{\phi
\theta }^{\phi }=\cot {\theta }\,.  \label{eq:appcon}
\end{align}

The associated spray $S$ of the above $SO(3)$-invariant symmetric affine
connection is given by 
\begin{equation}  \label{Eq:Geodesic_S}
S=\dot{t}\partial+\dot{r}\partial_r+\dot{\theta}\partial_\theta+\dot{\phi}%
\partial_\phi-2G^t\dot{\partial}_t-2G^r\dot{\partial}_r-2G^\theta\dot{%
\partial}_\theta-2G^\phi\dot{\partial}_\phi,
\end{equation}
where the coefficients $G^i$ are given by 
\begin{align*}
G^t=&\frac{1}{2}\left( k_1\dot{t}^2+2k_2 \dot{t}\dot{r}+k_3 \dot{r}^2+k_7%
\dot{\theta}^2+k_7\dot{\phi}^2 \sin^2\theta \right), \\
G^r=&\frac{1}{2}\left( k_4\dot{t}^2+2k_6 \dot{t}\dot{r}+k_5 \dot{r}^2+k_{10}%
\dot{\theta}^2+k_{10}\dot{\phi}^2 \sin^2\theta\right), \\
G^\theta=&\frac{1}{2}\left( 2k_8 \dot{t}\dot{\theta}+2k_{9}\dot{\theta}\dot{r%
}-\dot{\phi}(2k_{11}\dot{t}+2k_{12}\dot{r}+\dot{\phi} \cos\theta)\sin\theta
\right) , \\
G^\phi=&\frac{1}{2}\left( 2k_8 \dot{t}\dot{\phi}+2k_{9}\dot{\phi}\dot{r}+%
\dot{\phi}\dot{\theta}\cot\theta \right)+\frac{\dot{\theta}}{\sin\theta}%
(k_{11}\dot{t}+k_{12}\dot{r}).
\end{align*}

The curvature tensor $R^{c}{}_{ab}\dot{\partial}_{c}=\left[ \delta
_{a},\delta _{b}\right] $ of the induced nonlinear (Ehresmann) connection on 
$TM$ given by the coefficients $N^{a}{}_{b}=\Gamma _{bc}^{a}\dot{x}^{c}$,
can be locally expressed as 
\begin{equation*}
R^{a}{}_{bc}=\delta _{c}N^{a}{}_{b}-\delta _{b}N^{a}{}_{c}=\partial
_{c}N^{a}{}_{b}-N^{d}{}_{c}\dot{\partial}_{d}N^{a}{}_{b}-\partial
_{b}N^{a}{}_{c}+N^{d}{}_{b}\dot{\partial}_{d}N^{a}{}_{c}\,.
\end{equation*}%
where: 
\begin{equation}
\begin{array}{llll}
R_{~tr}^{t}=a_{1}\dot{t}+a_{2}\dot{r} & R_{~tr}^{r}=a_{3}\dot{t}+a_{4}\dot{r}
& R_{~tr}^{\theta }=a_{5}\dot{\theta} & R_{~tr}^{\varphi }=a_{5}\dot{\varphi}
\\ 
R_{~t\theta }^{t}=a_{6}\dot{\theta} & R_{~t\theta }^{r}=a_{7}\dot{\theta} & 
R_{t\theta }^{\theta }=a_{8}\dot{t}+a_{9}\dot{r} & R_{t\theta }^{\varphi }=0
\\ 
R_{~t\varphi }^{t}=a_{6}\dot{\varphi}\sin ^{2}\theta & R_{~t\varphi
}^{r}=a_{7}\dot{\varphi}\sin ^{2}\theta & R_{~t\varphi }^{\theta }=0 & 
R_{t\varphi }^{\varphi }=a_{8}\dot{t}+a_{9}\dot{r} \\ 
R_{~r\theta }^{t}=a_{10}\dot{\theta} & R_{~r\theta }^{r}=a_{11}\dot{\theta}
& R_{~r\theta }^{\theta }=a_{12}\dot{t}+a_{13}\dot{r} & R_{~r\theta
}^{\varphi }=0 \\ 
R_{~r\varphi }^{t}=a_{10}\dot{\varphi}\sin ^{2}\theta & R_{~r\varphi
}^{r}=a_{11}\dot{\varphi}\sin ^{2}\theta & R_{~r\varphi }^{\theta }=0 & 
R_{~r\varphi }^{\varphi }=a_{12}\dot{t}+a_{13}\dot{r} \\ 
R_{~\theta \varphi }^{t}=0 & R_{~\theta \varphi }^{r}=0 & R_{~\theta \varphi
}^{\theta }=-a_{14}\dot{\varphi}\sin ^{2}\theta & R_{~\theta \varphi
}^{\varphi }=a_{14}\dot{\theta}\,.%
\end{array}
\label{eq:appcurv}
\end{equation}%
The coefficients $a_{i}$ above are functions of $t$ and $r$, as follows: 
\begin{equation*}
\begin{split}
a_{1}& =k_{1,r}-k_{2,t}+k_{3}k_{4}-k_{2}k_{6}\,, \\
a_{2}& =k_{2,r}-k_{3,t}+k_{2}^{2}+k_{3}k_{6}-k_{1}k_{3}-k_{2}k_{5}\,, \\
a_{3}& =k_{4,r}-k_{6,t}+k_{1}k_{6}+k_{4}k_{5}-k_{2}k_{4}-k_{6}^{2}\,, \\
a_{4}& =k_{6,r}-k_{5,t}+k_{2}k_{6}-k_{3}k_{4}\,, \\
a_{5}& =k_{8,r}-k_{9,t}\,, \\
a_{6}& =-k_{7,t}+k_{7}k_{8}-k_{1}k_{7}-k_{2}k_{10}\,, \\
a_{7}& =-k_{10,t}+k_{8}k_{10}-k_{4}k_{7}-k_{6}k_{10}\,, \\
a_{8}& =-k_{8,t}+k_{1}k_{8}+k_{4}k_{9}-k_{8}^{2}\,, \\
a_{9}& =-k_{9,t}+k_{2}k_{8}+k_{6}k_{9}-k_{8}k_{9}\,, \\
a_{10}& =-k_{7,r}+k_{7}k_{9}-k_{2}k_{7}-k_{3}k_{10}\,, \\
a_{11}& =-k_{10,r}+k_{9}k_{10}-k_{6}k_{7}-k_{5}k_{10}\,, \\
a_{12}& =-k_{8,r}+k_{2}k_{8}+k_{6}k_{9}-k_{8}k_{9}\,, \\
a_{13}& =-k_{9,r}+k_{3}k_{8}+k_{5}k_{9}-k_{9}^{2}\,, \\
a_{14}& =1+k_{7}k_{8}+k_{9}k_{10}\,,
\end{split}%
\end{equation*}%
where the subscripts $_{,t}$ and $_{,r}$ means partial differentiation with
respect to $t$ and $r$ respectively.

The curvature components $R^{a}{}_{bcd}$ of the affine connection can be
obtained by $\dot{x}$-differentiation from $R_{~cd}^{a}$, as: 
\begin{equation}
R_{b~cd}^{~a}=\dot{\partial}_{b}R_{~cd}^{a};
\end{equation}%
e.g., $a_{1}=R_{t~tr}^{~t}$ etc.

\end{document}